\documentclass[12pt, two]{article}
\usepackage{}

\usepackage{CJK}
\usepackage{amsthm}
\usepackage{amsmath,amssymb,mathrsfs}
\usepackage{amsfonts}
\usepackage{yhmath}
\usepackage{graphics}
\usepackage{amssymb}
\usepackage{epsfig}
\usepackage{latexsym}
\usepackage{indentfirst}
\usepackage[top=1in, bottom=1in, left=1.25in, right=1.25in]{geometry}
\numberwithin{equation}{section}
\usepackage{CJK}
\allowdisplaybreaks

\newtheorem{theorem}{Theorem}[section]
\newtheorem{lemma}{Lemma}[section]
\newtheorem{proposition}{Proposition}[section]
\newtheorem{corollary}{Corollary}[section]
\newtheorem{remark}{Remark}[section]

\numberwithin{equation}{section} %¹«Ê½Ëæ½Ú±àºÅ

\title{
$L^2$ curvature pinching theorems and  vanishing theorems on complete Riemannian manifolds
\footnotetext[0]{ ${}^*$Supported by NSFC grant No. 11271071 and LMNS, Fudan.}
\footnotetext[0]{ ${}^{**}$Supported by NSFC grant No. 11401099.}
\footnotetext[0]{ ${}^{***}$Supported by NSF Award No DMS-1447008 and the OU Arts and Sciences TAP Fund.}
}

\author{Yuxin Dong$^{*}$, Hezi Lin$^{**}$ and Shihshu Walter Wei$^{***}$}
\date{}
\begin{document}

%\begin{CJK*}{GBK}{kai}

\maketitle
\begin{abstract}
In this paper, by  using monotonicity formulas for vector bundle-valued $p$-forms satisfying the conservation law, we first obtain general $L^2$  global rigidity theorems  for locally conformally flat (LCF) manifolds with constant  scalar curvature, under curvature pinching  conditions. Secondly, we prove  vanishing results for $L^2$ and some non-$L^2$ harmonic $p$-forms on LCF manifolds, by assuming that the underlying manifolds satisfy  pointwise or integral curvature conditions. Moreover, by a Theorem of Li-Tam for harmonic functions, we show that the underlying manifold must  have only   one end. Finally, we obtain Liouville theorems for $p$-harmonic functions  on LCF manifolds under pointwise Ricci curvature conditions.
\end{abstract}

\section{Introduction}

In the study of Riemannian geometry,  locally conformally flat  manifolds  play an important
role. Let us recall that an $n$-dimensional Riemannian manifold $(M^n, g)$  is said to be locally conformally flat (LCF) if it admits a coordinate covering $\{U_\alpha, \varphi_\alpha\}$ such that the map $\varphi_\alpha: (U_\alpha, g_{\alpha})\rightarrow (S^n, g_0)$ is a conformal map, where $g_0$ is the standard metric on $S^n$.
 A locally conformally flat manifold may be regarded as a higher dimensional generalization of a Riemann surface. But not every higher dimensional manifold admits a locally conformally flat structure, and it is  an interesting problem to give a good classification of locally conformally flat manifolds.
By assuming various geometric situations, many partial classification results have been given. (see, for examples,  \cite{CH, Che, CZ, PRS1, PV, SY, XZ, Zh}, etc.)

In the first part, we use the stress-energy tensor to study the rigidity of LCF manifolds.  In \cite{DW}, the authors   presented a unified method to establish monotonicity formulas
and vanishing theorems for vector-bundled valued $p$-forms satisfying a conservation law,
by means of the stress-energy tensors of various energy functionals in geometry and physics.
 Later, the authors in \cite{DL} established similar monotonicity formulas by using various exhaustion functions. As applications, they  proved the Ricci flatness of a K\"ahler manifold with constant scalar curvature under growth conditions for the Ricci form, and obtained Bernstein type theorems for submanifolds in Euclidean spaces with parallel mean curvature under growth conditions on the second fundamental form.
In this paper, we attempt  to use monotonicity formulas to study rigidity properties of    LCF metric with constant scalar curvature. For these aims, we may interpret the Riemannian (resp. Ricci) curvature tensor as a $2$-form (resp. $1$-form) with values in the bundle of symmetric endomorphisms of $T(M)$ endowed with its canonical structure of Riemannian vector bundle. For   LCF manifolds with constant scalar curvature,    the  $1$-forms  corresponding to the Ricci curvature tensor and to the traceless Ricci curvature tensor also satisfy conservation laws. Hence we can establish monotonicity formulas for those one forms,  from which $L^2$ curvature pinching theorems are deduced.

On the other hand,
it is an interesting problem in geometry and topology to find sufficient conditions
on a LCF manifold $M$ for the vanishing  of  harmonic forms.  When $M$ is compact, the Hodge theory states that the space of harmonic $p$-forms on $M$ is isomorphic to its $p$-th de Rham cohomology group.
In \cite{Bo}, Bourguignon  proved that  a compact, $2m$-dimensional, LCF manifold of positive scalar curvature has no non-zero harmonic $m$-forms, hence its $m$-th Betti number $\beta_m=0$.
 Guan, Lin and Wang \cite{GLW} obtained a cohomology vanishing theorem on compact LCF manifolds under a positivity assumption on the Schouten tensor. For the non-compact case, the Hodge theory is no longer true in general.  However, it is known that $L^2$ Hodge theory remains valid for complete non-compact manifolds.  Hence it is important to investigate $L^2$ harmonic forms.  In \cite{PRS2}, Pigola,  Rigoli and  Setti  showed a vanishing result for bounded harmonic forms of middle degree on complete  non-compact LCF manifolds, by adding suitable conditions on scalar curvature and volume growth.
 In \cite{Li}, Lin proved some vanishing and finiteness theorems for $L^2$ harmonic $1$-forms on complete non-compact LCF manifolds under integral  curvature pinching conditions.

 Since the Riemannian curvature of a LCF manifold can be expressed by its Ricci curvature and scalar curvature, we can compute explicitly the Weitzenb\"{o}ck formula for harmonic $p$-forms. Based on this  formula, together with $L^2$-Sobolev inequality or weighted Poincar\'{e}
inequality, we shall establish vanishing results for $L^2$ harmonic $p$-forms under various $L^n$-integral  curvature or pointwise curvature pinching conditions. In particular, we show that if the  Ricci tensor is sufficiently near  zero in the integral sense, then $H^p(L^2(M))= \{0\}$ for all $0 \leq p \leq n$, where $H^p(L^2(M))$ denote the space of all $L^2$ harmonic $p$-forms on  $M$. Moreover,  according to the nonexistence of nontrivial $L^2$ harmonic $1$-forms, we  deduce that $M$ has only one end by Li-Tam's harmonic functions theory.

Finally we also consider $p$-harmonic functions on LCF manifolds. When the scalar curvature of a  LCF manifold is negative,  it is known that a weighted   Poincar\'{e} inequality holds. Hence we can use the results of Chang-Chen-Wei \cite{CCW} to derive some Liouville theorems for $p$-harmonic functions, by assuming pointwise Ricci curvature bounds.

\section{Preliminaries}
Let $(M, g)$ be a complete manifold of dimension $n\geq 3$. Let $R_{ijkl}$ and $W_{ijkl}$ denote respectively the components of the Riemannian curvature tensor and  the Weyl curvature tensor  of $(M, g)$.  A fundamental result in Riemannian geometry is that
\begin{align}
\nonumber W_{ijkl}=&  R_{ijkl}- \frac{1}{n-2}(R_{ik}\delta_{jl}- R_{il}\delta_{jk} + R_{jl}\delta_{ik} - R_{jk}\delta_{il})\\
&+ \frac{R}{(n-1)(n-2)}(\delta_{ik}\delta_{jl}- \delta_{il}\delta_{jk}), \label{Weyl}
\end{align}
where $R_{ik}$ and $R$ denote the Ricci tensor and the scalar curvature respectively.  The associated Schouten tensor $A$ with respect to $g$ is defined by
\begin{equation*}
A:=\frac{1}{n-2} \left(\mbox{Ric}- \frac{R}{2(n-1)}g \right).
\end{equation*}
 It is well known that
 if $n = 3$, then $W_{ijkl}=0$, and $(M^3, g)$ is locally conformally flat if and only if
the Schouten tensor  is  Codazzi, i.e., $A_{ik,j} - A_{ij,k}=0$, where $A_{ij}$ is the components of  the Schouten tensor $A$. If $n  \geq 4$, then $(M^n, g)$ is locally conformally flat if and only if the Weyl tensor  vanishes, i.e., $W_{ijkl}=0$.
The local conformal flatness and the equation (\ref{Weyl}) yield
\begin{align}
\nonumber R_{ijkl}=&  \frac{1}{n-2}(R_{ik}\delta_{jl}- R_{il}\delta_{jk} + R_{jl}\delta_{ik} - R_{jk}\delta_{il})
-\frac{R}{(n-1)(n-2)}(\delta_{ik}\delta_{jl}\\
&- \delta_{il}\delta_{jk}). \label{LCF-decomposition}
\end{align}
Thus, a locally conformally flat manifold has
 constant sectional curvature if and only if it is Einstein, that is, $\mbox{Ric}=\frac{R}{n}g$.
  As a consequence, by the Hopf classification theorem, space forms are the only locally conformally flat Einstein manifolds.

Suppose $R$ is constant, by (2.1) and the second Bianchi identities,  we immediately obtain that the Ricci tensor is  Codazzi, that is,
$R_{ij,k}= R_{ik,j}$.
Therefore, the traceless Ricci tensor
$E= \mbox{Ric} - \frac{R}{n}g$
is  Codazzi too.

In order  to get vanishing results for $L^2$ harmonic $p$-forms on LCF manifols, we  need the following
$L^2$-Sobolev inequality. It is known that a  simply connected, LCF manifold $M^{n}$ ($n\geq 3$)  has a conformal immersion into $\mathbb{S}^{n}$, and according to \cite{SY}, the Yamabe constant of $M^n$ satisfies $Q(M^n)=Q(\mathbb{S}^n)= \frac{n(n-2)\omega_n^{\frac{2}{n}}}{4}$, where $\omega_n$ is the volume of the unit sphere in $\mathbb{R}^n$. Therefore  the following inequality
\begin{equation}
Q(\mathbb{S}^n)\left(\int_M f^{\frac{2n}{n-2}}dv\right)^{\frac{n-2}{n}}  \leq  \int_M |\nabla f|^2dv +
            \frac{n-2}{4(n-1)}\int_M Rf^2dv \label{Yamabe}
\end{equation}
holds for all $f \in C_0^{\infty}(M)$. If we assume  $R \leq 0$, then it follows that
\begin{equation}
Q(\mathbb{S}^n)\left(\int_M f^{\frac{2n}{n-2}}dv\right)^{\frac{n-2}{n}}  \leq  \int_M |\nabla f|^2dv, \ \ \forall f \in C_0^{\infty}(M). \label{R-negative1}
\end{equation}
On the other hand, if
$\int_M |R|^{\frac n 2}dv < \infty$,
then we can choose a compact set $\Omega\subset M $  large enough such that
\begin{equation*}
\left(\int_{M\setminus \Omega} |R|^{\frac n 2}dv\right)^{\frac 2n} \leq \frac{4\epsilon(n-1) Q(\mathbb{S}^n)}{n-2}
\end{equation*}
for some $\epsilon$ satisfying $0< \epsilon< 1$.
By the H\"{o}lder inequality, the term involving the scalar curvature can be
absorbed into the left-hand side of (\ref{Yamabe}) to yield
\begin{equation*}
(1-\epsilon)Q(\mathbb{S}^n)\left(\int_{M\setminus \Omega} f^{\frac{2n}{n-2}}dv\right)^{\frac{n-2}{n}}  \leq  \int_{M\setminus \Omega} |\nabla f|^2dv,   \ \ \forall  f \in C^\infty_0(M\setminus \Omega).
\end{equation*}
From the work of G. Carron \cite{Ca}  (one can also consult Theorem 3.2 of \cite{PST}), the following $L^2$-Sobolev inequality
\begin{equation}
C_s\left(\int_{M} f^{\frac{2n}{n-2}}dv\right)^{\frac{n-2}{n}}  \leq  \int_{M} |\nabla f|^2dv, \ \ \forall  f \in C^\infty_0(M) \label{finite-scalar}
\end{equation}
holds  for some uniform constant $C_s>0$, which implies a uniform lower bound on the volume of geodesic balls
\begin{equation}
\mbox{vol}(B_x(\rho))\geq C \rho^n,\ \ \forall x\in M \label{v-growth}
\end{equation}
for some constant $C>0$. Therefore, each end of $M$ has infinite volume.

\section{Monotonicity formulas for curvature tensor and vanishing results}
Let $(M, g)$  be a Riemannian manifold and $\xi:E\rightarrow M$ be a smooth Riemannian vector bundle
   over $M$   with compatible connection $\nabla^E$. Set $A^p(\xi)=\Gamma (\Lambda^pT^*M \otimes E)$ the space of smooth
   $p$-forms on $M$ with values in the vector bundle $\xi:E\rightarrow M$. The exterior covariant
   differentiation $d^\nabla:A^p(\xi)\rightarrow A^{p+1}(\xi)$ relative to  $\nabla^E$ is defined by
   \begin{equation*}
   (d^\nabla\omega)(X_1,\cdots,X_{p+1})=\overset{p+1}{\underset{i=1}{\sum}}(-1)^{i+1}(\nabla_{X_i}
   \omega)(X_1,\cdots,\widehat{X_i},\cdots,X_{p+1}).
   \end{equation*}
   The codifferential operator $\delta^\nabla:A^p(\xi)\rightarrow A^{p-1}(\xi)$ characterized as the adjoint of $d^\nabla$ if $M$ is compact or $\omega$ has a compact support, and is defined by
 \begin{equation*}
(\delta^{\nabla}\omega)(X_1,\cdots,X_{p-1})=-\underset{i}{\sum}(\nabla_{e_i}\omega)(e_i,X_1,\cdots,X_{p-1}),
\end{equation*}
where $\{e_i\}$ is an orthonormal basis of $TM$.
The energy functional of  $\omega\in A^p(\xi)$ is defined by
   $E(\omega)=\frac{1}{2}\int_M |\omega|^2dv_g$.
  Its stress-energy tensor is
\begin{equation}
S_{\omega} (X,Y)= \frac{|\omega|^2}{2}g(X,Y)-(\omega \odot \omega)(X,Y), \label{stress-energy}
\end{equation}
where $\omega \odot \omega\in \Gamma(A^p(\xi)\otimes A^p(\xi)) $ is a symmetric tensor defined by
\begin{equation}
(\omega \odot \omega)(X,Y)=\langle i_X\omega,i_Y\omega \rangle. \label{symmetric-tensor}
\end{equation}
Here $i_X\omega\in A^{p-1}(\xi)$ denotes the interior multiplication by $X\in \Gamma(TM)$.
 The divergence of $S_{\omega}$ is given by (cf. \cite{Xi, Ba})
\begin{equation}
(\mbox{div}S_\omega)(X)=\langle \delta^\nabla\omega,i_X\omega \rangle + \langle i_Xd^\nabla\omega,\omega \rangle.
\label{divergence}
\end{equation}
We say that a $2$-tensor field $T\in \Gamma(T^*M\otimes T^*M)$ is a Codazzi tensor if $T$ satisfies
\begin{equation*}
(\nabla_Z T)(X,Y)=(\nabla_YT)(X,Z)
\end{equation*}
for any vector field $X$, $Y$ and $Z$. One may regard $T \in \Gamma(T^*M \otimes T^*M)$ as a $1$-form $\omega_T$ with values in $T^*M$ as follows
\begin{equation}
\omega_T(X)=T(\cdot, X), \label{1-form}
\end{equation}
that is, $\omega_T \in A^1(T^*M)$. Note that the covariant derivative of $\omega_T$ is given by
\begin{align}
\nonumber \big ((\nabla_X\omega_T)(Y)\big )(e) =& \bigg (\nabla_X\big (\omega_T(Y)\big )- \omega_T(\nabla_XY)\bigg )(e)\\
\nonumber =& \bigg (\nabla_X\big ( \omega_T(Y)\big )\bigg )(e)  - T(e, \nabla_XY)\\
\nonumber =& \nabla_X\big (\omega_T(Y)(e)\big ) - \omega_T(Y)(\nabla_Xe) - T(e, \nabla_XY)\\
\nonumber =& \nabla_X \big (T(e, Y)\big ) - T(\nabla_Xe, Y) - T(e, \nabla_XY)\\
=& (\nabla_X T)(e, Y) \label{CT}
\end{align}
for any $X$, $Y\in \Gamma(TM)$ and $e \in T_xM$. Therefore $T$ is a Codazzi tensor if and only if
\begin{equation}
(\nabla_X \omega_T)(Y) = (\nabla_Y \omega_T)(X). \label{Codazzi-tensor}
\end{equation}
\begin{lemma}
The $2$-tensor field $T$ is a Codazzi tensor if and only if $d^\nabla \omega_T =0$.
\end{lemma}
\begin{proof}
By the definition of $d^\nabla$, we have
\begin{equation*}
(d^\nabla \omega_T)(X, Y) = (\nabla_X \omega_T)(Y)- (\nabla_Y \omega_T)(X), \ \ \forall X, Y \in TM.
\end{equation*}
Thus,  for any $X, Y \in TM$, $(\nabla_X \omega_T)(Y)=(\nabla_Y \omega_T)(X)$ is equivalent to $d^\nabla \omega_T=0$.
\end{proof}
\begin{remark}
There are many well-known examples of Codazzi tensors. These include any constant scalar multiple of the metric, and more generally any parallel self-adjoint $(1,1)$ tensor, such as the second fundamental form of submanifolds with parallel mean curvature in a space of constant sectional curvature. Furthermore, the Ricci tensor of a Riemannian manifold $M$ is Codazzi if and only if the curvature tensor of $M$ is harmonic. This is the case, for example, $M$
is an Einstein manifold.
\end{remark}

 Now we compute the codifferentiation of $\omega_T$. Choose an orthonormal frame field $\{ e_i\}^n_{i=1}$ around a point $x \in M$
 such that $(\nabla e_i)_x =0$. By (\ref{CT}), one gets
 \begin{equation}
\delta^{\nabla} \omega_T = - \underset{i=1}{\overset{n}{\sum}} (\nabla_{e_i}\omega_T)( e_i)
         = -\underset{i=1}{\overset{n}{\sum}} (\nabla_{e_i}T)(\cdot, e_i). \label{delta}
\end{equation}
\begin{lemma}
Let $T$ be a symmetric Codazzi $2$-tensor field. If $tr T$ is constant, then $\delta^\nabla \omega_T =0$.
\end{lemma}
\begin{proof}
Note that $\delta^\nabla \omega_T \in \Gamma(T^*M)$. For any vector $X \in TM$, we get from (\ref{delta}) that
\begin{align*}
(\delta^\nabla \omega_T)(X) =& - \underset{i=1}{\overset{n}{\sum}}(\nabla_{e_i}T)(X, e_i).
\end{align*}
Since $T$ is Codazzi, it follows that
\begin{align*}
(\delta^\nabla \omega_T)(X) = - \underset{i=1}{\overset{n}{\sum}}(\nabla_{X}T)(e_i, e_i)
= - X (\underset{i=1}{\overset{n}{\sum}} T(e_i, e_i))=0.
\end{align*}
\end{proof}

Therefore, by (\ref{1-form}), Lemma 3.1 and Lemma 3.2, we have the following proposition:
\begin{proposition}\label{pro-3.1}
Suppose $T$ is a symmetric Codazzi 2-tensor with constant trace. Then $\omega_T$ satisfies a conservation law, that is, $\mbox{div}S_{\omega_T} =0$ as defined in $(\ref{divergence})$.
\end{proposition}

For any given vector field $X$, there corresponds to a  dual one form $\theta_X$ such that
\begin{equation*}
\theta_X(Y)=g(X,Y), \ \ \forall\,  Y\in \Gamma(TM).
\end{equation*}
The covariant derivative of $\theta_X $ gives a 2-tensor field $\nabla\theta_X$:
\begin{equation*}
(\nabla\theta_X)(Y,Z)=(\nabla_Z\theta_X)(Y)=g(\nabla_Z X,Y), \ \  \forall\, Y,Z\in TM.
\end{equation*}
If $X=\nabla\psi$ is the gradient of some  smooth function $\psi$ on $M$, then $\theta_X=d\psi$ and $\nabla\theta_X=\mbox{Hess}(\psi)$.
A direct computation yields (cf. \cite{Xi} or Lemma 2.4 of \cite{DW}):
\begin{equation}
\mbox{div}(i_XS_{\omega})=\langle S_{\omega},\nabla\theta_X \rangle+(\mbox{div} S_{\omega})(X), \ \ \forall\, X \in TM. \label{divergence-2}
\end{equation}
Let $D$ be any bounded domain of $M$ with $C^1$ boundary. By (\ref{divergence-2}) and using the divergence theorem, we immediately have
\begin{equation}
\int_{\partial D}S_{\omega}(X,\nu)ds_g=\int_{D}\big (\langle S_{\omega},\nabla \theta_X \rangle +(\mbox{div} S_{\omega})(X)\big )dv_g,
\end{equation}
where $\nu$ is the unit outward normal vector field along $\partial D$. In particular, if $\omega$ satisfies the conservation
law, i.e. $\mbox{div} S_{\omega}=0$,
 then
\begin{equation}
\int_{\partial D}S_{\omega}(X,\nu)ds_g= \int_{D}\langle S_{\omega},\nabla \theta_X \rangle dv_g. \label{integral-formula}
\end{equation}

Let $r(x)$ be the geodesic distance function of $x$ relative to some fixed point $x_0$ and $B_{x_0}(r)$ be the geodesic ball centered at $x_0$
with radius $r$. Denote by $\lambda_1(x)\leq \lambda_2(x)\leq \cdots \leq  \lambda_n(x)$  the eigenvalues of $\mbox{Hess}(r^2)$.
Let
\begin{equation}
\tau(p)=\frac{1}{2}\underset{x\in M}{\inf}\{\lambda_1(x)+\cdots+\lambda_{n-p}(x)-\lambda_{n-p+1}(x)-\cdots - \lambda_n(x)\}
\end{equation}
be a function depending  only on the integer $p$, $1 \leq p \leq n$.
\begin{proposition}
\label{pro-monot}
 Let $(M, g)$ be an $n$-dimensional complete Riemannian manifold
 and let $\xi:E \rightarrow M$ be a Riemannian vector bundle on $M$.  If $\tau(p)>0$ and $\omega\in A^p(\xi)$ satisfies the conservation law,
that is, $\mbox{div} S_{\omega}=0$, then
\begin{equation}
\frac{1}{\rho_1^ {\sigma}} \int_{B_{x_0}(\rho_1)}
|\omega|^2dv \leq \frac{1}{\rho_2^{\sigma}}
\int_{B_{x_0}(\rho_2)}|\omega|^2dv
\end{equation}
for any $0<\rho_1 \leq \rho_2$ and $ 0< \sigma \leq \tau(p)$.
\end{proposition}
\begin{proof}
The proof is similar to that of \cite{DW}. We will provide the argument
here for completeness of the paper.
Take a smooth vector field $X= r\nabla r$ on $M$.
Obviously, $\frac{\partial}{\partial r}$ is an outward unit normal vector field
along $\partial B_{x_0}(r)$.
Take an orthonormal  basis $\{e_i\}_{i=1}^n$ which diagonalizes  $\mbox{Hess}(r^2)$,
then
\begin{align}
\nonumber \langle S_\omega,\nabla\theta_X \rangle =& \frac{1}{2}\overset{n}{\underset{i,j=1}{\sum}}
 S_\omega(e_i,e_j)\mbox{Hess}(r^2)(e_i,e_j)\\
\nonumber =& \frac{1}{4} \overset{n}{\underset{i,j=1}{\sum}}|\omega|^2 \mbox{Hess}(r^2)(e_i,e_j)\delta_{ij}
  -\frac{1}{2}\overset{n}{\underset{i,j=1}{\sum}}(\omega \odot \omega)(e_i,e_j) \mbox{Hess}(r^2)(e_i,e_j)\\
  =& \frac{|\omega|^2}{4} \overset{n}{\underset{i=1}{\sum}} \lambda_i
  -\frac{1}{2}\overset{n}{\underset{i=1}{\sum}}(\omega \odot \omega)(e_i,e_i) \lambda_i. \label{SW}
\end{align}
For the second term, by (\ref{symmetric-tensor}), we have
\begin{align*}
\overset{n}{\underset{i=1}{\sum}}(\omega \odot \omega)(e_i,e_i)\lambda_i
=& \overset{n}{\underset{s=1}{\sum}}\langle i_{e_s}\omega, i_{e_s}\omega \rangle \lambda_{s}\\
=& \frac{1}{p!}\overset{p}{\underset{j=1}{\sum}}\overset{}{\underset{i_1,\cdots, i_p}{\sum}}\langle \omega(e_{i_1}, \cdots), \omega(e_{i_1}, \cdots) \rangle \lambda_{i_j}\\
\leq& \frac{1}{p!}\overset{}{\underset{i_1,\cdots, i_p}{\sum}}\langle \omega(e_{i_1}, \cdots), \omega(e_{i_1}, \cdots) \rangle \overset{n}{\underset{j=n-p+1}{\sum}} \lambda_{j}\\
=& |\omega|^2 \overset{n}{\underset{j=n-p+1}{\sum}} \lambda_{j}.
\end{align*}
Substituting into (\ref{SW}), it follows that
\begin{align}
 \langle S_\omega,\nabla\theta_X \rangle
 \geq \frac{|\omega|^2}{4}(\lambda_1+\cdots+\lambda_{n-p}-\lambda_{n-p+1}-\cdots - \lambda_n). \label{right}
\end{align}
By the definition of $S_{\omega}$, we have
\begin{align}
\nonumber S_{\omega}(X,\frac{\partial}{\partial r})=& \frac{|\omega|^2}{2}g(X,\frac{\partial}{\partial r})-(\omega \odot \omega)(X,\frac{\partial}{\partial r})\\
 \nonumber =& \frac{1}{2}r|\omega|^2g(\frac{\partial}{\partial r},\frac{\partial}{\partial r})
-\frac{1}{2}r |i_{\frac{\partial}{\partial r}}\omega |^2\\
\leq& \frac{r|\omega|^2}{2}   \quad  \mbox{on} \ \partial B_{x_0}(r). \label{left}
 \end{align}
Since  $\mbox{div}S_{\omega}=0$,  we get from (\ref{integral-formula}), (\ref{right}) and (\ref{left}) that
 $$\frac{1}{2}\underset{x\in M}{\inf}(\lambda_1+\cdots+\lambda_{n-p}-\lambda_{n-p+1}-\cdots - \lambda_n) \int_{ B_{x_0}(r)} |\omega|^2 dv \leq  r\int_{\partial B_{x_0}(r)} |\omega|^2 dv.$$
Using co-area formula, we have
\begin{eqnarray*}
\tau(p) \int_{B_{x_0}(r)}|\omega|^2dv  \leq r\frac{d}{dr} \int_{B_{x_0}(r)}|\omega|^2dv,
\end{eqnarray*}
thus
\begin{eqnarray*}
 \frac{\frac{d}{dr} \int_{B_{x_0}(r)}|\omega|^2dv}
        {\int_{B_{x_0}(r)}|\omega|^2dv}\geq
 \frac{\sigma}{r}
\end{eqnarray*}
for any $\sigma \leq \tau(p)$. Integrating the above formula on $[\rho_1,\rho_2]$ yields
  $$\frac{1}{\rho_1^ { \sigma}} \int_{B_{x_0}(\rho_1)}
|\omega|^2dv \leq \frac{1}{\rho_2^{\sigma}}
\int_{B_{x_0}(\rho_2)}|\omega|^2dv.$$
\end{proof}

In the following, we shall use Proposition \ref{pro-monot} to deduce  monotonicity formulas and vanishing results for  the curvature tensor of LCF manifolds.  For this purpose, we  collect the following Lemmas.

\begin{lemma} \label{lem-Hessian} $(\cite{GW, DW, HLRW})$
Let $(M, g)$ be a complete Riemannian manifold with a pole  $x_0$ and let $r$ be the distance function relative to $x_0$.
Denote by $K$ the radial curvature of $M$.

$(i)$ If $-\frac{A}{(1+r^2)^{1+\epsilon}} \leq K \leq \frac{B}{(1+r^2)^{1+\epsilon}}$ with $\epsilon>0$, $A \geq 0$, $0\leq B < 2\epsilon$,
then
$$\frac{1-\frac{B}{2\epsilon}}{r} [g-dr \otimes dr] \leq \mbox{Hess}(r) \leq \frac{e^{\frac{A}{2\epsilon}}}{r}[g-dr \otimes dr]. $$

$(ii)$ If $-\frac{a^2}{1+r^2} \leq K \leq \frac{b^2}{1+r^2}$ with $a\geq 0$, $b^2 \in[0,1/4]$, then
$$\frac{1+\sqrt{1-4b^2}}{2r} [g-dr \otimes dr] \leq \mbox{Hess}(r) \leq \frac{1+\sqrt{1+4a^2}}{2r}[g-dr \otimes dr].$$

$(iii)$ If $-\alpha^2\leq K \leq -\beta^2$ with $\alpha> 0$, $\beta> 0$, then
$$\beta \coth(\beta r)[g-dr \otimes dr] \leq \mbox{Hess}(r) \leq \alpha \coth(\alpha r)[g-dr \otimes dr].$$

\end{lemma}

Using Lemma \ref{lem-Hessian}, by a direct calculation  we have the following result.

 \begin{lemma} \label{lem-alpha}$(\cite{DW, DL})$
Let $M^n$ be a complete  manifold of  dimension $n$ with a pole $x_0$. Assume that the radial curvature of $M$
satisfies one of the following  conditions:

$(i)$ $-\frac{A}{(1+r^2)^{1+\epsilon}} \leq K \leq \frac{B}{(1+r^2)^{1+\epsilon}}$ with $\epsilon>0$, $A \geq 0$, $0 \leq B < 2\epsilon$
and $n -(n-1)\frac{B}{2\epsilon}-2pe^{A/2\epsilon} > 0$;

$(ii)$ $-\frac{a^2}{1+r^2} \leq K \leq \frac{b^2}{1+r^2}$ with $a\geq 0$, $b^2 \in[0,1/4]$ and
$1+\frac{n-1}{2}(1+\sqrt{1-4b^2})-p(1+\sqrt{1+4a^2}) > 0$;

$(iii)$ $-\alpha^2\leq K \leq -\beta^2$ with  $\alpha> 0$, $\beta> 0$ and $(2n-3)\beta-2(p-1)\alpha \geq 0$.
Then
\begin{equation*}
 \tau(p)\geq \sigma(p)=: \begin{cases}
n -(n-1)\frac{B}{2\epsilon}-2pe^{A/2\epsilon} & \text{if $K$ satisfies $(i)$,} \\
1+\frac{n-1}{2}(1+\sqrt{1-4b^2})-p(1+\sqrt{1+4a^2}) & \text{if $K$ satisfies $(ii)$,}\\
n-1-(p-1)\frac{\alpha}{\beta} & \text{if $K$ satisfies $(iii)$.}
\end{cases}
\end{equation*}
\end{lemma}

 Let $(M^n, g)$ be a Riemannian manifold of dimension $n$, and let $V\rightarrow M$ be the vector bundle of skew-symmetric endomorphisms of $TM$ endowed with its canonical Riemannian structure. Then the curvature tensor $Rm$ can be seen as a $V$-valued $2$-form and thus the second Bianchi identity can be equivalently expressed as  $dRm=0$. Actually, using moving frame method, we may compute
\begin{align*}
(d R_{ij})_{klm}=& R_{ijlm, k} -R_{ijkm, l}+ R_{ijkl, m}\\
=& R_{ijlm, k} + R_{ijmk, l}+ R_{ijkl, m}=0.
\end{align*}

\begin{lemma}
 Let $(M^n, g)$, $n\geq 3$, be a LCF Riemannian manifold with constant scalar curvature. Then the curvature tensor $Rm$ is a harmonic $V$-valued $2$-form and thus $Rm$ satisfies a conservation law, that is, $\mbox{div} S_{Rm} = 0$ as defined in $(\ref{divergence})$.
\end{lemma}
\begin{proof}
We only need to prove that $d Rm=0$ and $\delta Rm=0$. We have already pointed out that the first property is just the second Bianchi identity. In terms of the condition that $M$ is a LCF manifold with constant scalar curvature, we find that
\begin{align*}
(\delta Rm)_{jkl}=& R_{ijkl, i}= \nabla_i R_{ijkl}= \nabla_i R_{klij}\\
=& -\nabla_k R_{ijli} - \nabla_l R_{ijik}\\
=& \nabla_k R_{jl} - \nabla_l R_{jk}=0.
\end{align*}
\end{proof}
\begin{remark}
By the relation $(\ref{Weyl})$, the Weyl curvature tensor $W$ of an Einstein manifold is also a harmonic $V$-valued $2$-form. Thus, $W$ also satisfies a conservation law.
\end{remark}

For the Ricci tensor  $\mbox{Ric}$, we can consider   $\mbox{Ric}$  to be  a $1$-form $\mbox{Ric}^\sharp$ with values in the tangent vector bundle at every point $x \in M$, that is, for every $X \in T_xM$, $\mbox{Ric}^\sharp(X)$ satisfies
\begin{equation*}
\langle \mbox{Ric}^\sharp(X), Y \rangle = \mbox{Ric}(X, Y), \ \forall\,  Y \in T_xM.
\end{equation*}
$E^\sharp$ satisfies $\langle E^\sharp(X), Y \rangle = E(X, Y),  \ \forall Y \in T_xM$, where E is the traceless Ricci tensor given by $E= \mbox{Ric} - \frac{R}{n}g$.
Thus if $M$  is a conformally flat Riemannian manifold with constant scalar curvature,
then by Proposition \ref{pro-3.1}, $\mbox{Ric}^\sharp$ and $E^\sharp$ satisfy conservation laws, that is, $\mbox{div} S_{\mbox{Ric}^\sharp} = 0$  and $\mbox{div} S_{E^\sharp} = 0$ as defined in (\ref{divergence}).
Let $|\operatorname{Ric}|$ be the norm of Ricci tensor $\operatorname{Ric}$ and $|E|$ be the norm of the traceless Ricci tensor $E$ given by $|\operatorname{Ric}| = ( \overset{n}{\underset{i,j=1}{\sum}} R_{ij}^2 )^{\frac {1}{2}}\, $ and $|E| = \bigg ( \overset{n}{\underset{i,j=1}{\sum}}(R_{ij} - \frac {R}{n} \delta _{ij})^2 \bigg)^{\frac {1}{2}}\, $ respectively. Summarizing the previous discussions, we have the following results.
\begin{theorem} \label{thm-Rie}
 Let $(M^n, g)$ be a complete, locally conformally flat Riemannian manifold with constant scalar curvature. Assume that the radial curvature of $M$ satisfies the conditions of Lemma \ref{lem-alpha}.
 Then for any $0<\rho_1 \leq \rho_2$,

 \begin{equation*}
\frac{1}{\rho_1^ { \sigma(2)}} \int_{B_{x_0}(\rho_1)}
|Rm|^2dv \leq \frac{1}{\rho_2^{\sigma(2)}}
\int_{B_{x_0}(\rho_2)}|Rm|^2dv,
\end{equation*}
and
\begin{equation*}
\frac{1}{\rho_1^ { \sigma(1)}} \int_{B_{x_0}(\rho_1)}
|\mbox{Ric}|^2dv \leq \frac{1}{\rho_2^{\sigma(1)}}
\int_{B_{x_0}(\rho_2)}|\mbox{Ric}|^2dv,
\end{equation*}
and
\begin{equation*}
\frac{1}{\rho_1^ { \sigma(1)}} \int_{B_{x_0}(\rho_1)}
|E|^2dv \leq \frac{1}{\rho_2^{\sigma(1)}}
\int_{B_{x_0}(\rho_2)}|E|^2dv,
\end{equation*}
where
\begin{equation*}
\sigma(p)\big{|}_{p=1,2}= \begin{cases}
n -(n-1)\frac{B}{2\epsilon}-2pe^{A/2\epsilon} & \text{if $K$ satisfies $(i)$,} \\
1+\frac{n-1}{2}(1+\sqrt{1-4b^2})-p(1+\sqrt{1+4a^2}) & \text{if $K$ satisfies $(ii)$,}\\
n-1-(p-1)\frac{\alpha}{\beta} & \text{if $K$ satisfies $(iii)$.}
\end{cases}
\end{equation*}
\end{theorem}

\begin{corollary} \label{cor3.3}
Let $M^n$, $n\geq 3$, be a complete, locally conformally flat Riemannian manifold with a pole $x_0$ and with  constant scalar curvature. Assume   the radial curvature of $M$
satisfies one of the following  conditions:

$(i)$ $-\frac{A}{(1+r^2)^{1+\epsilon}} \leq K \leq \frac{B}{(1+r^2)^{1+\epsilon}}$ with $\epsilon>0$, $A \geq 0$, $0 \leq B < 2\epsilon$ and $n -(n-1)\frac{B}{2\epsilon}-2e^{A/2\epsilon}>0$;

$(ii)$ $-\frac{a^2}{1+r^2} \leq K \leq \frac{b^2}{1+r^2}$ with $a\geq 0$, $b^2 \in[0,1/4]$ and
$1+\frac{n-1}{2}(1+\sqrt{1-4b^2})-(1+\sqrt{1+4a^2})> 0$.\\
Assume further that
\begin{equation*}
\int_{B_{x_0}(\rho)} |\mbox{Ric}|^2 dv_g = o(\rho^{\sigma(1)})  \ \ \mbox{as} \ \ \rho \rightarrow +\infty.
\end{equation*}
Then $M$ is flat.
\end{corollary}
\begin{proof}
It follows from Theorem \ref{thm-Rie} and the growth condition for $|\mbox{Ric}|^2$ that $(M,g)$ is Ricci-flat. Then (\ref{LCF-decomposition}) implies immediately that $Rm=0$.
\end{proof}
\begin{remark}
If the locally conformally flat manifold $M$ is scalar flat, then its sectional curvature can be controlled only by its Ricci curvature. Hence in this case, the curvature conditions $(i)$ and $(ii)$ in Corollary \ref{cor3.3} can be replaced by corresponding Ricci curvature.
\end{remark}
\begin{corollary}
Let $M^n$, $n\geq 3$,  be a complete, locally conformally flat Riemannian manifold  with  constant scalar curvature. Assume  $M$ has a pole $x_0$ and its radial curvature satisfies
$-\alpha^2\leq K \leq -\beta^2$ with  $\alpha \geq \beta> 0$.
If
\begin{equation*}
\int_{B_{x_0}(\rho)} |E|^2 dv_g = o(\rho^{n-1}) \ \ \mbox{as} \ \ \rho \rightarrow +\infty,
\end{equation*}
then $M$ is of constant curvature $\frac{R}{n(n-1)}$.
\end{corollary}

\begin{remark}
It is easy to see that the above rigidity results for LCF manifolds with constant scalar curvature also hold for the curvature tensor and the Weyl curvature tensor of Einstein manifolds.
\end{remark}

\section{Vanishing theorems for $L^2$ harmonic $p$-forms on LCF manifolds}

Let $(M^n, g)$ be a complete, locally conformally flat Riemannian manifold, and let $\triangle$ be the Hodge Laplace-Beltrami operator of $M^n$ acting on the space of differential $p$-forms.
 Given two $p$-forms $\omega$ and $\theta$, we define a pointwise inner product
\begin{equation*}
\langle \omega, \theta \rangle=\underset{i_1, \cdots, i_p=1}{\overset{n}{\sum}} \omega(e_{i_1}, \cdots, e_{i_p}) \theta(e_{i_1}, \cdots, e_{i_p}).
\end{equation*}
Here we are omitting the normalizing factor $\frac{1}{p!}$. The Weitzenb\"{o}ck formula (\cite{Wu}) gives
\begin{equation}
\triangle= \nabla^2 -\mathcal {R}_p, \label{Weit}
\end{equation}
where $\nabla^2 $ is the Bochner Laplacian and $\mathcal {R}_p$ is an endomorphism depending upon the curvature tensor of $M^n$. Using an orthonormal basis $\{\theta^1, \ldots, \theta^n\}$ dual to $\{e_1, \ldots, e_n\}$,  the curvature term $\mathcal {R}_p$ can be  expressed as
\begin{equation*}
\langle \mathcal {R}_p(\omega), \omega \rangle = \langle \underset{j,k=1}{\overset{n}{\sum}} \theta^k \wedge i_{e_j}R(e_k, e_j)\omega, \omega \rangle
\end{equation*}
for any $p$-form $\omega$.
Let $\omega$ be any harmonic $p$-form, which may be expressed in a local coordinate system as
\begin{equation*}
\omega=\alpha_{i_1,\cdots,i_p}dx^{i_1}\wedge \cdots \wedge dx^{i_p}.
\end{equation*}
By (\ref{Weit}), we deduce that
\begin{align}
\nonumber \frac{1}{2}\triangle |\omega|^2 =&|\nabla \omega|^2 +\langle \underset{j,k=1}{\overset{n}{\sum}} \theta^k \wedge i_{e_j}R(e_k, e_j)\omega, \omega \rangle\\
  =&|\nabla \omega|^2 + p F(\omega), \label{laplacian1}
\end{align}
where
\begin{equation*}
 F(\omega) = R_{ij}\alpha^{ii_2 \cdots i_p}\alpha^j_{i_2 \cdots i_p} -\frac{p-1}{2}R_{ijkl}\alpha^{iji_3 \cdots i_p}\alpha^{kl}_{i_3 \cdots i_p}.
\end{equation*}
Substituting  (\ref{LCF-decomposition}) into the above equality, we obtain
\begin{align}
\frac{1}{2}\triangle |\omega|^2 =&|\nabla \omega|^2 + p\Big[\frac{n-2p}{n-2}R_{ij}\alpha^{ii_2 \cdots i_p}\alpha^j_{i_2 \cdots i_p}+ \frac{(p-1)}{(n-1)(n-2)}R |\omega|^2 \Big] \label{F0}\\
 =& |\nabla \omega|^2 + p\Big[\frac{n-2p}{n-2}\Big(R_{ij} -\frac{R}{n}\delta_{ij}  \Big)\alpha^{ii_2 \cdots i_p}\alpha^j_{i_2 \cdots i_p} + \frac{n-p}{n(n-1)}R |\omega|^2 \Big]. \label{F}
\end{align}

Using the  method of Lagrange multipliers, one has the following lemma.
\begin{lemma} $\cite{SW}$ \label{lagrange}
Let $(a_{ij})_{n\times n}$ be a real symmetric matrix with $\overset {n}{\underset{i=1}{\sum}} a_{ii}=0$, then
\begin{equation*}
\overset {n}{\underset{i, j=1}{\sum}} a_{ij}x_i x_j \geq -\sqrt{\frac{n-1}{n}}\big(\overset {n}{\underset{i, j=1}{\sum}} a^2_{ij} \big)^{\frac{1}{2}}\sum_{i=1}^{n} x^2_i
\end{equation*}
where $x_i \in \mathbb{R}$.
\end{lemma}
By Lemma \ref{lagrange}, it follows from (\ref{F}) that
\begin{align}
 \frac{1}{2}\triangle |\omega|^2 \geq |\nabla \omega|^2 -\frac{p|n-2p|}{n-2}\sqrt{\frac{n-1}{n}}|E||\omega|^2+ \frac{p(n-p)}{n(n-1)}R |\omega|^2, \label{laplacian2}
\end{align}
where $|E|$ is the norm of the traceless Ricci tensor $E\, .$ On the other hand, we have
\begin{equation*}
 \frac{1}{2}\triangle |\omega|^2  = |\omega| \triangle |\omega| + |\nabla|\omega||^2.
\end{equation*}
Combining this with (\ref{laplacian2}) and the refined  Kato's inequality (\cite{CGH})
  \begin{equation}
 |\nabla \omega|^2 - |\nabla|\omega||^2 \geq  K_p |\nabla|\omega||^2,   \label{Koto}
 \end{equation}
  where
\begin{equation*}
K_p =\begin{cases}
\frac{1}{n-p}  & \text{if $1 \leq p \leq n/2 $},\\
\frac{1}{p}    & \text{if $n/2 \leq p \leq n-1 $},
\end{cases}
\end{equation*}
we conclude that
\begin{equation}
 |\omega| \triangle |\omega| \geq  K_p |\nabla|\omega||^2 -\frac{p|n-2p|}{n-2}\sqrt{\frac{n-1}{n}}|E||\omega|^2+ \frac{p(n-p)}{n(n-1)}R |\omega|^2. \label{laplacian3}
\end{equation}

Now, using the  inequality (\ref{laplacian2}) to compact locally conformally flat Riemannian manifold, we have the following theorem, generalizing Bourguignon's result \cite{Bo}, for the case $R(x) > 0$ for every $x \in M$ and $p=m= \frac{n}{2}\, .$
\begin{theorem}
\label{the-compact}
 Let $(M^n, g)$ be a compact  locally conformally flat Riemannian manifold satisfying
\begin{equation}
R (x) \geq \sqrt{\frac{n-1}{n}}\frac{n(n-1)|n-2p|}{(n-p)(n-2)}|E|(x) \label{R-E}
\end{equation}
for every  $x \in M$, $1 \leq p \leq n$. Assume that $(\ref{R-E})$ is  strict at some point. Then  the Betti number $\beta_p(M)= 0$. In particular, if  $M$ is a $2m$-dimensional compact LCF Riemannian manifold with nonnegative scalar curvature  $R\geq  0$, and $R  > 0$ holds at some point, then $\beta_m(M)= 0$.
\end{theorem}
\begin{proof}
For any given harmonic $p$-form $\omega$, we have via (\ref{laplacian2}) and the hypothesis (\ref{R-E}) on the scalar curvature $R$,
\begin{align}
\frac{1}{2}\triangle |\omega|^2 \ge |\nabla \omega|^2 +\Big[ \frac{p(n-p)}{n(n-1)}R-\frac{p|n-2p|}{n-2}\sqrt{\frac{n-1}{n}}|E|\Big] |\omega|^2 \geq 0. \label{thm01}
\end{align}
 By the compactness of   $M$  and the maximum principle, $|\omega|=\mbox{constant}$. Substituting this into (\ref{thm01}) and using the  hypothesis on $R$ again,  we have $\omega=0$. Therefore, by Hodge's Theorem, $\beta_p(M)= 0$.
\end{proof}

\begin{remark}
It is well known that  a compact orientable conformally flat Riemannian
manifold with positive Ricci curvature must satisfy $\beta_p(M)= 0$ for all $1\leq p \leq  n-1$.
\end{remark}
\begin{theorem}
\label{thm-Ric}
 Let $(M^n, g)$, $n \geq 3$,  be a complete non-compact, simply connected,    locally conformally flat Riemannian manifold. Then there exists a constant $C>0$ such that if
\begin{equation}
\int_M |\mbox{Ric}|^{\frac{n}{2}}dv< C, \label{Lp-Ric}
\end{equation}
then for every $0 \leq p \leq n$, every harmonic $p$-form  $\omega$ on $M$ with
$\underset{r\rightarrow \infty}{\liminf} \frac{1}{r^2} \int_{B_{x_0}(r)} |\omega|^2 dv =0$
vanishes identically. In particular,  $H^p(L^2(M))= \{0\}$.
\end{theorem}
\begin{proof}
Let $\omega $ be a harmonic $p$-form on $M$ for $0 \leq p \leq n$ with $\underset{r\rightarrow \infty}{\liminf} \,  \frac{1}{r^2} \int_{B_{x_0}(r)} |\omega|^2 dv =0$.
 When $1 \leq p \leq n-1$, by (\ref{F0}) and using the fact that $R^2 \leq n|\mbox{Ric}|^2 $, we have
\begin{align}
\nonumber \frac{1}{2}\triangle |\omega|^2 \geq  |\nabla \omega|^2- \frac{p |n-2p|}{n-2}|\mbox{Ric}| |\omega|^2 -\frac{p(p-1)\sqrt{n}}{(n-1)(n-2)}|\mbox{Ric}||\omega|^2.
\end{align}
Combining this with (\ref{Koto}), we deduce that
\begin{align}
 |\omega| \triangle |\omega| + \frac{p}{n-2}\bigg (|n-2p| +\frac{(p-1)\sqrt{n}}{n-1}\bigg )|\mbox{Ric}||\omega|^2 \geq K_p|\nabla|\omega||^2. \label{Ricci-laplacian}
\end{align}
 Fix a point $x_0\in M$ and let $\rho(x)$ be the geodesic distance on $M$ from $x_0$ to $x$. Let us choose  $\eta \in C^{\infty}_0(M)$ satisfying
\begin{equation*}
\eta(x)= \begin{cases}
1  & \text{if\  $\rho(x) \leq r$},\\
0   & \text{if \  $2 r < \rho(x) $}
\end{cases}
\end{equation*}
and
\begin{equation}
  | \nabla \eta|(x) \leq \frac{1}{r} \ \ \mbox{if} \ \  r < \rho(x) \leq 2r \label{cut-off}
\end{equation}
for $r>0$. Multiplying (\ref{Ricci-laplacian}) by $\eta^2$ and integrating by parts  over $M$, we obtain
\begin{align}
\nonumber 0 \leq& \int_M (\eta^2 |\omega| \triangle |\omega|  - K_p \eta^2 |\nabla |\omega||^2)dv
+ \frac{p}{n-2}\bigg (|n-2p| +\frac{(p-1)\sqrt{n}}{n-1}\bigg )\int_M |\mbox{Ric}|\eta^2|\omega|^2 dv\\
\nonumber =& -2\int_M \eta |\omega| \langle \nabla\eta, \nabla |\omega| \rangle dv - (1+ K_p) \int_M \eta^2 |\nabla |\omega||^2dv \\
& + \frac{p}{n-2}\bigg (|n-2p| +\frac{(p-1)\sqrt{n}}{n-1}\bigg )\int_M |\mbox{Ric}|\eta^2|\omega|^2 dv. \label{thm40}
\end{align}
By the hypothesis (\ref{Lp-Ric}), we have
\begin{equation*}
\int_M |R|^{\frac n 2}dv \leq n^{n/4}\int_M |\mbox{Ric}|^{\frac n 2}dv< \infty,
\end{equation*}
which implies that  the $L^2$-Sobolev inequality (\ref{finite-scalar}) holds  for some constant $C_s>0$. Hence it follows from (\ref{finite-scalar}) and the H\"{o}lder  inequality that
\begin{align}
\nonumber \int_M |\mbox{Ric}|\eta^2|\omega|^2 dv \leq& \left(\int_M |\mbox{Ric}|^{\frac{n}{2}}dv\right)^{\frac{2}{n}}
     \left(\int_M (\eta |\omega|)^{\frac{2n}{n-2}}dv\right)^{\frac{n-2}{n}}\\
\nonumber \leq&  R(\eta) \int_M |\nabla (\eta |\omega|)|^2dv\\
\nonumber =&  R(\eta)\int_M (\eta^2|\nabla |\omega||^2+ |\omega|^2|\nabla \eta|^2) dv \\
& +2 R(\eta) \int_M \eta|\omega|\langle \nabla\eta, \nabla |\omega|\rangle dv, \label{L-2Ric}
\end{align}
where $R(\eta) =\frac {1}{C_s}\left(\int_M |\mbox{Ric}|^{\frac{n}{2}}dv\right)^{\frac{2}{n}}$.
Substituting (\ref{L-2Ric}) into (\ref{thm40}) yields
\begin{align}
\nonumber 0 \leq&  (2A-2) \int_M \eta |\omega| \langle \nabla\eta, \nabla |\omega| \rangle dv - (1+ K_p- A) \int_M \eta^2 |\nabla |\omega||^2dv + A\int_M  |\omega|^2|\nabla \eta|^2 dv\\
\nonumber \leq& (- 1- K_p+ A + |A-1|\epsilon) \int_M \eta^2 |\nabla |\omega||^2dv + \left(A + \frac{|A-1|}{\epsilon}\right)\int_M  |\omega|^2|\nabla \eta|^2 dv
\end{align}
for all $\epsilon>0$, where
\begin{equation*}
A=\frac{p}{n-2}\bigg (|n-2p| +\frac{(p-1)\sqrt{n}}{n-1}\bigg ) R(\eta) .
\end{equation*}
Now let us choose  the integral bound $C$ in (\ref{Lp-Ric}) satisfying
\begin{equation*}
C^{\frac{2}{n}}=\frac{n-2}{p}\bigg (|n-2p| +\frac{(p-1)\sqrt{n}}{n-1}\bigg )^{-1}(1+ K_p)C_s.
\end{equation*}
Then we can take sufficiently small $\epsilon >0$ such that $1+ K_p- A- |A-1|\epsilon >0$. Therefore,
\begin{align*}
\nonumber    (1+ K_p- A- |A-1|\epsilon) \int_{B_{x_0}(r)} |\nabla |\omega||^2dv \leq&
(1+ K_p- A-|A-1|\epsilon) \int_M \eta^2 |\nabla |\omega||^2dv \\
\nonumber  \leq& \left(A + \frac{|A-1|}{\epsilon}\right)\int_M  |\omega|^2|\nabla \eta|^2 dv\\
\leq& \frac{\epsilon A + |A-1|}{\epsilon r^2}\int_{B_{x_0}(2r)}  |\omega|^2 dv.
\end{align*}
 Letting $r \rightarrow \infty$, we have $\nabla |\omega|=0$ on $M$, i.e., $|\omega|$ is constant. Since $\underset{r\rightarrow \infty}{\liminf}  \frac{1}{r^2} \int_{B_{x_0}(r)} |\omega|^2 dv =0$ and the volume growth (\ref{v-growth}) implies $\frac{\mbox{vol}(B_{x_0}(r))}{r^2} \geq  C r^{n-2} \rightarrow \infty$ as $r \rightarrow \infty$, we conclude   that $\omega=0$.

When $p=0$,  let $f$ be a harmonic function with $\underset{r\rightarrow \infty}{\liminf} \,  \frac{1}{r^2} \int_{B_{x_0}(r)} |\omega|^2 dv =0$. According to \cite{W1}, $f$ is  constant. Since $\frac{\mbox{vol}(B_{x_0}(r))}{r^2} \geq  C r^{n-2}$, we have $f=0$.
  When $p=n$, we consider $\ast \omega$, where $\ast$ is the Hodge Star. Then $\ast \omega$ is a harmonic function
with $|\omega|= |\ast \omega|$. By the previous result, $\ast \omega=0$ and so is $\omega=0$. It follows that
$H^p(L^2(M))=\{0\}$ for all $0 \leq p \leq n$. This completes the proof.
\end{proof}
\begin{remark}
Since the  constant $C_s$ in the   Sobolev inequality $(\ref{finite-scalar})$ can not be explicitly computed,
we can't also give the explicit value of $C$ in $(\ref{Lp-Ric})$.
\end{remark}

\begin{theorem}
\label{thm-global-pinching}
 Let $(M^n, g)$, $n\geq 3$, be a complete non-compact, simply connected, locally conformally flat Riemannian manifold with $R \geq 0$. Assume that
 \begin{equation}
\Big(\int_M |E|^{\frac{n}{2}}dv\Big)^{\frac{2}{n}} < C(p), \label{gap}
\end{equation}
where $C(p)=\frac{(n-2)\sqrt{n}}{p|n-2p|\sqrt{n-1}} \min \Big\{1+ K_p, \frac{4p(n-p)}{n(n-2)}\Big\} Q(\mathbb{S}^n)$ for every $1\leq p \leq n-1$ but $p\neq \frac{n}{2}$.  Then every harmonic $p$-form  $\omega$ on $M$ with
$\underset{r\rightarrow \infty}{\liminf} \,  \frac{1}{r^2} \int_{B_{x_0}(r)} |\omega|^2 dv =0$
vanishes identically. In particular,  $H^p(L^2(M))= \{0\}$ for $1\leq p \leq n-1$ but $p\neq \frac{n}{2}$.
\end{theorem}
\begin{proof}
Let $\omega $ be a harmonic $p$-form on $M$  with $\underset{r\rightarrow \infty}{\liminf} \,  \frac{1}{r^2} \int_{B_{x_0}(r)} |\omega|^2 dv =0$.  Let $\eta \in C^{\infty}_0(M)$ be a smooth function on $M$ with compact support.
Multiplying (\ref{laplacian3}) by $\eta^2$ and integrating over $M$, we obtain
\begin{align}
\nonumber \int_M \eta^2 |\omega| \triangle |\omega|dv \geq&  K_p \int_M \eta^2 |\nabla |\omega||^2 dv
                 - \frac{p|n-2p|}{n-2}\sqrt{\frac{n-1}{n}}\int_M |E|\eta^2 |\omega|^2 dv\\
  &+ \frac{p(n-p)}{n(n-1)}\int_M R \eta^2 |\omega|^2 dv. \label{integral}
\end{align}
Integrating by parts and using the Cauchy-Schwarz inequality gives
\begin{align}
\nonumber \int_M \eta^2 |\omega| \triangle |\omega|dv =& -2\int_M \eta |\omega| \langle \nabla\eta, \nabla |\omega| \rangle dv- \int_M \eta^2 |\nabla |\omega||^2dv \\
 \nonumber \leq &  (b- 1 )\int_M \eta^2 |\nabla |\omega||^2dv + \frac{1}{b}\int_M |\omega|^2 |\nabla \eta|^2dv
\end{align}
for all $b>0$.
Substituting the above inequality into (\ref{integral}) yields
\begin{align}
\nonumber ( 1+ K_p-b )\int_M \eta^2 |\nabla |\omega||^2dv  \leq &   \frac{1}{b}\int_M |\omega|^2 |\nabla \eta|^2dv
    +\frac{p|n-2p|}{n-2}\sqrt{\frac{n-1}{n}}\int_M |E|\eta^2 |\omega|^2 dv\\
    &        - \frac{p(n-p)}{n(n-1)}\int_M R \eta^2 |\omega|^2 dv. \label{414}
\end{align}
 On the other hand, using (\ref{Yamabe}) together with  the H\"{o}lder   and  Cauchy-Schwarz inequalities, we have
\begin{align*}
\nonumber \int_M |E|\eta^2 |\omega|^2 dv \leq& \left(\int_{\mbox{supp}(\eta)} |E|^{\frac{n}{2}}dv\right)^{\frac{2}{n}} \left(\int_M (\eta
           |\omega|)^{\frac{2n}{n-2}}dv\right)^{\frac{n-2}{n}}\\
\nonumber \leq&  \frac{1}{Q(\mathbb{S}^n)}\left(\int_{\mbox{supp}(\eta)} |E|^{\frac{n}{2}}dv \right)^{\frac{2}{n}}
    \int_M \Big[|\nabla (\eta |\omega|)|^2 + \frac{n-2}{4(n-1)}R\eta^2 |\omega|^2 \Big]dv \\
\nonumber=& T(\eta)\int_M \Big[\eta^2|\nabla |\omega||^2+|\omega|^2|\nabla \eta|^2 +\frac{n-2}{4(n-1)} R\eta^2 |\omega|^2\Big] dv\\
\nonumber &+2T(\eta) \int_M \eta |\omega|\langle \nabla\eta, \nabla |\omega| \rangle dv\\
 \leq & T(\eta)\int_M \Big[(1+ \gamma)\eta^2|\nabla |\omega||^2+ \big(1+ \frac{1}{\gamma}\big)|\omega|^2|\nabla \eta|^2 +\frac{n-2}{4(n-1)} R\eta^2 |\omega|^2\Big] dv
\end{align*}
for all $\gamma>0$, where $\mbox{supp}(\eta)$ is the support of $\eta$ on $M$, and $T(\eta)= \frac{1}{Q(\mathbb{S}^n)}(\int_{\mbox{supp}(\eta)} |E|^{\frac{n}{2}}dv)^{\frac{2}{n}}$.
Substituting the above inequality into (\ref{414}), we conclude that
\begin{align}
B \int_M \eta^2 |\nabla |\omega||^2dv \leq C \int_M  |\omega|^2|\nabla \eta|^2 dv+ D \int_M R \eta^2 |\omega|^2 dv, \label{cacciop}
\end{align}
where
\begin{align*}
 B=& 1 + K_p- b-\frac{p|n-2p|}{n-2}\sqrt{\frac{n-1}{n}}T(\eta)(1+\gamma),\\
C=& \frac{1}{b} + \frac{p|n-2p|}{n-2}\sqrt{\frac{n-1}{n}}T(\eta)\big(1+ \frac{1}{\gamma}\big),\\
D=& \frac{p|n-2p|}{4(n-1)}\sqrt{\frac{n-1}{n}}T(\eta)-\frac{p(n-p)}{n(n-1)}.
\end{align*}
It follows from  the hypothesis (\ref{gap}) that for $1\leq p \leq n-1$ but $p\neq \frac{n}{2}\, ,$
 \begin{equation*}
T(\eta)= \frac{1}{Q(\mathbb{S}^n)}\Big(\int_{\mbox{supp}(\eta)} |E|^{\frac{n}{2}}dv\Big)^{\frac{2}{n}}
  < \frac{n-2}{p|n-2p|}\sqrt{\frac{n}{n-1}} \min \Big\{1+ K_p, \frac{4p(n-p)}{n(n-2)}\Big \},
\end{equation*}
which implies that $D< 0$ and $1+ K_p- \frac{p|n-2p|}{n-2}\sqrt{\frac{n-1}{n}}T(\eta)>0$. Hence we can choose $\gamma$  and $b$ small enough such that
\begin{equation*}
B= 1 + K_p- b-\frac{p|n-2p|}{n-2} \sqrt{\frac{n-1}{n}}T(\eta)(1+\gamma)>0.
\end{equation*}

Let $\eta$ be the cut-off function defined by (\ref{cut-off}).
Substituting $\eta$ into  (\ref{cacciop}) and noting the hypothesis $R \geq 0$, we have
\begin{align*}
  B \int_{B_{x_0}(r)}  |\nabla |\omega||^2dv \leq&  B \int_M \eta^2 |\nabla |\omega||^2dv \\
   \leq& \frac{C}{r^2} \int_{B_{x_0}(2r)} |\omega|^2  dv+ D \int_{B_{x_0}(2r)} R |\omega|^2 dv. \label{416}
\end{align*}
Letting $r \rightarrow \infty$, and noting  $\underset{r\rightarrow \infty}{\liminf}\,  \frac{1}{r^2} \int_{B_{x_0}(r)} |\omega|^2 dv =0$, we conclude that
\begin{equation*}
\nabla |\omega|=0 \ \ \ \mbox{and} \ \ \ R|\omega|=0
\end{equation*}
on $M\, .$ Hence, $|\omega|=\mbox{constant}$. If $|\omega|$ is not identically
zero, then $R=0$,  which implies that the $L^2$-Sobolev  inequality (\ref{R-negative1}) holds,  and $\frac{\mbox{vol}(B_{x_0}(r))}{r^2} \geq  C r^{n-2} \rightarrow \infty$ as $r \rightarrow \infty$. This would contradict $\underset{r\rightarrow \infty}{\liminf} \,  \frac{1}{r^2} \int_{B_{x_0}(r)} |\omega|^2 dv =0$. Therefore,  $\omega=0$. It follows that
$H^p(L^2(M))=\{0\}$  for $1\leq p \leq n-1$ but $p\neq \frac{n}{2}$. This completes the proof.
\end{proof}

For the middle degree case, we deduce the following vanishing theorem without  assumptions on $E$.

\begin{theorem}
\label{thm-gap-m}
 Let $(M^n, g)$, $n=2m>3$, be a complete non-compact, simply connected, locally conformally flat Riemannian manifold with $R\geq 0$. Then every harmonic $m$-form  $\omega$ on $M$ with
$\underset{r\rightarrow \infty}{\liminf} \,  \frac{1}{r^2} \int_{B_{x_0}(r)} |\omega|^2 dv =0$
vanishes identically. In particular,  $H^m(L^2(M))= \{0\}$.
\end{theorem}
\begin{proof}
Taking $p=m=\frac{n}{2}$ in (\ref{laplacian3}), we have
\begin{equation}
 |\omega|\triangle |\omega| \geq \frac{1}{m}|\nabla|\omega||^2 +\frac{m}{2(2m-1)}R|\omega|^2. \label{middle-laplacian}
\end{equation}
Let $\eta$ be the cut-off function defined by (\ref{cut-off}). Multiplying (\ref{middle-laplacian}) by $\eta^2$  and integrating by parts  over $M$, we obtain
\begin{align*}
&\frac{m+1}{m} \int_M  |\nabla|\omega||^2 \eta^2 dv  + \frac{m}{2(2m-1)}\int_M R |\omega|^2 \eta^2 dv\\
&\leq    \frac{1}{2}\int_M \triangle |\omega|^2 \eta^2 dv \\
&= -2 \int_M \langle |\omega|\nabla \eta,  \eta \nabla |\omega| \rangle dv\\
&\leq  m \int_M |\omega|^2 |\nabla \eta|^2dv + \frac{1}{m} \int_M  |\nabla|\omega||^2 \eta^2 dv,
\end{align*}
which implies that
\begin{align*}
\int_{B_{x_0}(r)}  |\nabla|\omega||^2  dv  + \frac{m}{2(2m-1)}\int_{B_{x_0}(r)} R |\omega|^2 dv
 \leq &   m \int_M |\omega|^2 |\nabla \eta|^2dv\\
 \leq &  \frac{m}{r^2} \int_{B_{x_0}(2r)} |\omega|^2 dv.
\end{align*}
Having established this fact, the rest of the proof is completely analogous to that  of Theorem \ref{thm-global-pinching}.
\end{proof}
\begin{remark}
Pigola, Rigoli and Setti $\cite{PRS2}$ proved a vanishing theorem for bounded harmonic $m$-form on a $2m$-dimensional complete LCF manifold by putting some assumptions on scalar curvature and volume growth.
\end{remark}

Combining Theorems \ref{thm-global-pinching} and \ref{thm-gap-m}, we immediately have
\begin{corollary}
 Let $(M^n, g)$, $n\geq 3$, be a complete non-compact, simply connected, locally conformally flat Riemannian manifold with $R \geq 0$. Then there exists a positive constant $C$ such that if
 \begin{equation*}
\int_M |E|^{\frac{n}{2}}dv < C,
\end{equation*}
Then every harmonic $p$-form  $\omega$ on $M$ with
$\underset{r\rightarrow \infty}{\liminf}\,   \frac{1}{r^2} \int_{B_{x_0}(r)} |\omega|^2 dv =0$
vanishes identically for  every $1 \leq p \leq n-1$. In particular,  $H^p(L^2(M))= \{0\}$ for  every $1 \leq p \leq n-1$.
\end{corollary}

\begin{theorem}
\label{thm-m-R}
 Let $(M^n, g)$ be a complete non-compact, simply connected, locally conformally flat Riemannian manifold of dimension $n=2m>3$. Then there exists  $C>0$ such that if
 \begin{equation*}
\int_M |R|^{m}dv< C,
\end{equation*}
 then every harmonic $m$-form  $\omega$ on $M$ with
$\underset{r\rightarrow \infty}{\liminf}\,  \frac{1}{r^2} \int_{B_{x_0}(r)} |\omega|^2 dv =0$
vanishes identically. In particular,  $H^m(L^2(M))= \{0\}$.
\end{theorem}
\begin{proof}
 It follows from (\ref{middle-laplacian}) that
\begin{equation*}
 |\omega|\triangle |\omega| + \frac{m}{2(2m-1)}|R| |\omega|^2\geq \frac{1}{m}|\nabla|\omega||^2.
\end{equation*}
By an analogue argument of Theorem \ref{thm-global-pinching}, we immediately complete the proof.
\end{proof}

Let us recall that a Riemannian manifold $M$ is said to have nonnegative isotropic curvature if
 \begin{equation*}
R_{1313} + R_{1414} + R_{2323} + R_{2424} -2R_{1234}\geq 0
\end{equation*}
for every orthonormal $4$-frame $\{e_1, e_2, e_3, e_4\}$. From \cite{MN}, we know that if $M$ is conformally flat and has nonnegative isotropic curvature, then $F(\omega) \geq 0$ for any $2 \leq p \leq [\frac{n}{2}]$. Thus, it follows from the relations (\ref{laplacian1}) and (\ref{Koto}) that
 \begin{equation*}
 |\omega| \triangle |\omega|  \geq \frac{1}{n-p}|\nabla|\omega||^2.
 \end{equation*}
 Therefore, using the previous argument and the duality generated by the star operator $\ast$, we have the following result.
\begin{theorem}
\label{iso1}
% Let $(M^n, g)$,  $n\geq 4$, be a complete    locally conformally flat Riemannian manifold with infinite volume and
 % nonnegative isotropic curvature.
 %Then  for every $2 \leq p \leq n-2$, every harmonic $p$-form  $\omega$ on $M$ with
%$\liminf_{r\rightarrow \infty} \frac{1}{vol(B_{x_0}(r))} \int_{B_{x_0}(r)} |\omega|^2 dv =0$
%vanishes identically. In particular,  $H^p(L^2(M))= \{0\}$ for  all $2 \leq p \leq n-2$.
Let $(M^n, g)$,  $n\geq 4$, be a complete locally conformally flat Riemannian manifold with
  nonnegative isotropic curvature.
 Then for every $2 \leq p \leq n-2\, ,$ $(i)$ if $\underset{r\rightarrow \infty}{\liminf} \, \frac{\operatorname{vol} (B_{x_0}(r))}{r^2} > 0\, ,$ then every harmonic $p$-form  $\omega$ on $M$ with
$\underset{r\rightarrow \infty}{\liminf} \, \frac{1}{r^2} \int_{B_{x_0}(r)} |\omega|^2 dv =0$
vanishes identically. $(ii)$ if $M$ has infinite volume then  $H^p(L^2(M))= \{0\}\, .$ \end{theorem}

 For a LCF Riemannian manifold, a direct computation from (\ref{LCF-decomposition}) gives
\begin{equation*}
 R_{ijkl}=0
 \end{equation*}
 and
 \begin{equation*}
 R_{ijij}=\frac{1}{n-2}\left(\mbox{Ric}_{ii}+ \mbox{Ric}_{jj}- \frac{R}{n-1}\right)
 \end{equation*}
for all distinct $i,j, k, l$. If
\begin{equation*}
\mbox{Ric} \geq  \frac{R}{2(n-1),}
 \end{equation*}
 then $R_{ijij}\geq 0$,  and $M$ has nonnegative isotropic curvature. Applying Theorem \ref{iso1}, we have the following corollary.
\begin{corollary}
 Let $(M^n, g)$, $n\geq 4$, be a complete  non-compact  locally conformally flat Riemannian manifold. Assume that
 \begin{equation*}
\mbox{Ric}(x) \geq  \frac{1}{2(n-1)}R(x)
 \end{equation*}
 for all $x \in M$. Then   $H^p(L^2(M))= \{0\}$ for  all $2 \leq p \leq n-2$.
\end{corollary}
\begin{proof}
According to the  previous discussion,  $M$ is of  nonnegative Ricci curvature. Since $M$ is complete non-compact, we conclude from \cite{Ya} that $M$ has infinite volume. Hence the conclusion follows immediately from Theorem \ref{iso1} $(ii)$.
\end{proof}

For the four dimensional case, recall that an oriented Riemannian manifold of dimension $4$  is said to be half-conformally flat if  either the self-dual Weyl tensor $W^+=0$ or the anti-self-dual Weyl tensor $W^-=0$. Without loss of generality, we assume that $W^+=0$.

By the property of $W^-$, for any $k, l=1,2, 3,4$, we have
\begin{equation*}
W^-_{12kl}= -W^-_{34kl},\ \ \ W^-_{13kl}= -W^-_{42kl},\ \ \ W^-_{14kl}= -W^-_{23kl}.
\end{equation*}
Combining with the first Bianchi identity, we compute
\begin{align*}
&W^-_{1313} + W^-_{1414} +W^-_{2323}+W^-_{2424}- 2W^-_{1234}\\
&= -W^-_{4213}-W^-_{2314} -W^-_{1423} -W^-_{3124} -2 W^-_{1234}\\
&= -2W^-_{1342}-2W^-_{1423}-2 W^-_{1234}\\
&=0.
\end{align*}
Hence, the assumption $W^+=0$ and the relation (\ref{Weyl}) imply that
\begin{align*}
&R_{1313} + R_{1414} + R_{2323} + R_{2424} -2R_{1234}= \frac{1}{3}R.
\end{align*}
Therefore, from the proof of Theorem 2.1 in \cite{MW}, an analogous argument as Theorem \ref{iso1} yields
\begin{theorem}
 Let $(M^4, g)$ be a complete, half-conformally flat Riemannian manifold with $R\geq 0$ and with infinite volume. Then $H^2(L^2(M))= \{0\}$.
\end{theorem}

 If $M^n$ is a locally conformal flat manifold
with $R \leq  0$, then $M$ supports a weighted Poincar\'{e} inequality
 \begin{equation}
 \int_M \Big(|\nabla \phi|^2 -\frac{n-2}{4(n-1)}|R|\phi^2 \Big)dv \geq 0, \ \ \forall \phi \in C^{\infty}_0(M), \label{weighted-P}
\end{equation}
 which is equivalent to the nonnegative eigenvalue of the Schr\"{o}dinger operator $\triangle + \frac{n-2}{4(n-1)}|R|$. Thus under a lower bound condition of Ricci curvature, we can deduce the following vanishing theorem.
\begin{theorem}
\label{thm-pinching}
Let $(M^n, g)$, $n\geq 4$, be a complete, simply connected, locally conformally flat manifold with $R \leq 0$.  Suppose  the Ricci curvature of $M$ satisfies the lower bound
\begin{equation}
\mbox{Ric}(x) \geq \frac{(n-2)^2 - 4p(p-1)}{4p(n-1)(n-2p)} R(x) \label{Ric-lower}
\end{equation}
for $1 \leq p < [\frac{n}{2}]$ at every $x \in M$.  Then  every harmonic $p$-form  $\omega$ on $M$ with
$\underset{r\rightarrow \infty}{\liminf} \, \frac{1}{r^2} \int_{B_{x_0}(r)} |\omega|^2 dv =0$
vanishes identically. In particular,  $H^p(L^2(M))= \{0\}$.
\end{theorem}
\begin{proof}
Let $\omega $ be a harmonic $p$-form on $M$  with $\underset{r\rightarrow \infty}{\liminf} \, \frac{1}{r^2} \int_{B_{x_0}(r)} |\omega|^2 dv =0$. Substituting (\ref{Ric-lower}) into (\ref{F0}), and using (\ref{Koto}), we have
\begin{align}
\nonumber  |\omega| \triangle |\omega| \geq& \frac{1}{n-p}|\nabla |\omega||^2 +\frac{p(n-2p)}{n-2}\frac{(n-2)^2 - 4p(p-1)}{4p(n-1)(n-2p)} R |\omega|^2\\
\nonumber &+\frac{p(p-1)}{(n-1)(n-2)}R|\omega|^2\\
\geq& \frac{1}{n-p}|\nabla|\omega||^2 + \frac{n-2}{4(n-1)}R|\omega|^2. \label{4.24}
\end{align}
Let $\eta$ be the cut-off function defined by (\ref{cut-off}).
Choosing $\phi= \eta |\omega|$ in (\ref{weighted-P}), using (\ref{4.24}) and integrating by parts, we compute
\begin{align}
\nonumber0 \leq& \int_M \Big(|\nabla (\eta|\omega|)|^2 -\frac{n-2}{4(n-1)}|R| \eta^2 |\omega|^2 \Big)dv\\
\nonumber =& \int_M \Big(- \eta |\omega|\triangle (\eta |\omega|) -\frac{n-2}{4(n-1)}|R| \eta^2 |\omega|^2\Big)dv\\
\nonumber =& -\int_M \eta |\omega|(|\omega|\triangle \eta + \eta \triangle |\omega| +2 \langle \nabla \eta, \nabla |\omega|\rangle )dv -\frac{n-2}{4(n-1)}\int_M|R| \eta^2 |\omega|^2dv\\
\nonumber =& -\int_M \Big[\eta^2 |\omega| \triangle |\omega| +\frac{n-2}{4(n-1)}|R||\omega|^2 \Big]dv-2 \int_M \eta |\omega| \langle \nabla \eta, \nabla |\omega|\rangle dv\\
\nonumber & - \int_M |\omega|^2\eta \triangle \eta dv \\
\nonumber \leq& -\frac{1}{n-p}\int_M  \eta^2 |\nabla |\omega||^2dv +\int_M |\omega|^2|\nabla \eta|^2 dv \\
\leq& -\frac{1}{n-p}\int_{B_{x_0}(r)}   |\nabla |\omega||^2dv
+\frac{1}{r^2}\int_{B_{x_0}(2r)} |\omega|^2 dv.
\end{align}
Letting $r \rightarrow \infty$ and  using  $\underset{r\rightarrow \infty}{\liminf} \, \frac{1}{r^2} \int_{B_{x_0}(r)} |\omega|^2 dv =0$, we infer
\begin{equation*}
\nabla |\omega|=0.
\end{equation*}
Hence $\omega$ is constant. Since
$\frac{\mbox{vol}(B_{x_0}(r))}{r^2} \geq  C r^{n-2} \rightarrow \infty$ as $r \rightarrow \infty$ by the assumption  $R \leq 0$, we conclude that $\omega=0$.
\end{proof}

\section{Liouville theorems of $p$-harmonic functions on LCF manifolds
 with negative scalar curvature}

We recall a real-valued $C^{3}$ function $u$ on a Riemannian  $M$  is said to be \textit{strongly
}$p$-\textit{harmonic} if $u$ is a (strong) solution of the $p$-Laplace
equation
\begin{equation}
\begin{array}{l}
\Delta _{p}u:=\text{\textrm{div}}\left( |\nabla u|^{p-2}\nabla u\right) =0
\end{array}
\label{0.1}
\end{equation}%
for $p>1$.
A function $u\in W_{loc}^{1,p}\left( M\right) $ is said to be \textit{weakly
}$p$-\textit{harmonic} if
\begin{equation*}
\int_{M}\left\vert \nabla u\right\vert ^{p-2}\left\langle \nabla u,\nabla
\phi \right\rangle dv=0, \ \ \forall \phi \in C_{0}^{\infty }\left( M\right).
\end{equation*}%
It is well known that the $p$-Laplace equation (\ref{0.1}) arises as the Euler-Lagrange equation
of the $p$-energy  functional  $E_{p}(u) = \int_{M} |\nabla
u|^{p}\, dv\, . $

We say that $M$ supports a weighted Poincar\'{e}
inequality $\left( P_{\rho }\right) $, if there exists a positive function $%
\rho (x)$ a.e. on $M$ such that
\begin{equation*}
\quad  \left( P_{\rho }\right) \quad \quad \quad \int_{M}\rho\left( x\right) f^{2}\left( x\right) dv  \leq
\int_{M}\left\vert \nabla f \left( x\right) \right\vert ^{2}dv, \ \ \forall f \in W_{0}^{1.2}\left(
M\right).
\label{WP}
\end{equation*}

In \cite{CCW},  Chang,  Chen and  Wei introduce and study an approximate solution of the $p$-Laplace equation,
and a linearlization $\mathcal{L}_{\epsilon }$ of a perturbed $p$-Laplace
operator. They prove a Liouville type theorem for weakly $p$%
-harmonic functions with finite $p$-energy on a complete noncompact manifold
$M$ which supports a weighted Poincar\'{e} inequality $\left( P_{\rho }\right) $ and satisfies a
curvature assumption. This nonexistence result, when combined with an
existence theorem, implies that such an
$M$ has at most one $p$-hyperbolic end.  More precisely, the following is proved:
\medskip

\noindent
{\bf Theorem A} \cite{CCW}
\label{T1}{\it Let $M$ be a complete non-compact Riemannian $n$-manifold, $n\geq
2, $ supporting a weighted Poincar\'{e}
inequality $\left( P_{\rho }\right)$ with Ricci curvature
\begin{equation}
\begin{array}{lll}
\mbox{Ric}(x)  \geq  -\tau \rho (x)%
\end{array}
\label{Rs}
\end{equation}%
for all $x\in M,$ where $\tau $ is a constant such that
\begin{equation}
\begin{array}{l}
\tau <\frac{p-1+\kappa}{p^{2}},\text{ }%
\end{array}
\label{1.7}
\end{equation}%
in which $p>1\, ,$ and

\begin{equation}\kappa = \begin{cases}
 \max \{ \frac {1}{n-1}, \min \{ \frac{\left( p-1\right) ^{2}}{n},1\}\}\qquad & \operatorname{if}\quad p>2, \\
 \frac{\left( p-1\right) ^{2}}{n-1} \qquad & \operatorname{if}\quad 1 < p \le 2.
 \end{cases}
\end{equation}
Then every weakly $p$-harmonic function $u$ with finite $p$-energy $E_{p}$
is constant. Moreover, $M$ has at most one $p$-hyperbolic end.}
\bigskip

Moreover, a Liouville type
theorem for strongly $p$-harmonic functions with finite $q$-energy on
Riemannian manifolds is obtained:
\medskip

\noindent
{\bf Theorem B} \cite{CCW}
\label{T2}{\it Let $M$ be a complete non-compact Riemannian $n$-manifold, $n\geq
2, $ satisfying $\left( P_{\rho }\right),$ with Ricci curvature
\begin{equation}
\begin{array}{lll}
\mbox{Ric}(x)  \geq -\tau \rho (x)%
\end{array}
\label{Rs}
\end{equation}%
for all $x\in M,$ where $\tau $ is a constant such that
\begin{equation}
\begin{array}{l}
\tau <\frac{4\left( q-1+\kappa +b\right) }{q^{2}},\text{ }%
\end{array}
\label{1.7}
\end{equation}%
in which
\begin{equation}
\begin{array}{l}
\kappa =\min \{\frac{\left( p-1\right) ^{2}}{n-1},1\}\text{ and}\text{ }%
b=\min \{0,(p-2)(q-p)\},\text{ where }p>1.%
\end{array}%
\end{equation}%
Let $u\in C^{3}\left( M\right) $ be a strongly $p$-harmonic function with
finite $q$-energy $E_{q}\left( u\right) <\infty .$ \newline
$(I)$ Then $u$ is constant under each one of the following conditions:
\newline
$(1)$ $p=2$ and $q>\frac{n-2}{n-1},$ \newline
$(2)$ $p=4,$ $q>\max \left\{ 1,1-\kappa -b\right\} ,$ \newline
$(3)$ $p>2,$ $p\neq 4,$ and either%
\begin{equation*}
\begin{array}{l}
\max \left\{ 1,p-1-\frac{\kappa }{p-1}\right\} <q\leq \min \left\{ 2,p-\frac{%
\left( p-4\right) ^{2}n}{4\left( p-2\right) }\right\}%
\end{array}%
\end{equation*}%
or%
\begin{equation*}
\begin{array}{l}
\max \left\{ 2,1-\kappa -b\right\} <q,%
\end{array}%
\end{equation*}%
\newline
$(II)$ $u$ does not exist for $1<p<2$ and $q>2.$}
\bigskip

We recall in Sect.4, if $M$ is a locally conformal flat manifold
with scalar curvature $R < 0\, ,$ a.e., then $M$ supports a weighted Poincar\'{e} inequality (\ref{weighted-P}) or $\left( P_{\rho }\right) $ in which
$\rho = -\frac {n-2}{4(n-1)}R\, .$  Applying Theorems A and B, we have

\begin{theorem}
Let $(M^n, g)$, $n\geq 3$, be a complete non-compact, simply connected, locally conformal flat Riemannian manifold with scalar curvature $R < 0\, ,$ a.e. and Ricci curvature satisfying
\begin{equation}
\begin{array}{lll}
\mbox{Ric}(x)  \geq  a R(x)%
\end{array}
\label{Rs}
\end{equation}%
for all $x\in M,$ where $a $ is a constant such that
\begin{equation}
\begin{array}{l}
a <\frac {n-2}{4(n-1)} \cdot \frac{p-1+\kappa}{p^{2}},\text{ }%
\end{array}
\label{1.7}
\end{equation}%
in which $p>1\, ,$ and $\kappa$ is as in $(5.4)\, .$
Then every weakly $p$-harmonic function $u$ with finite $p$-energy $E_{p}$
is constant. Moreover, $M$ has at most one $p$-hyperbolic end.\end{theorem}

\begin{proof}
Since $M$ supports a weighted Poincar\'{e} inequality (\ref{weighted-P}) or $\left( P_{\rho }\right) $ in which
$\rho = -\frac {n-2}{4(n-1)}R\, ,$
the inequalities $(5.8)$ and $(5.9)$ are equivalent to the inequalities $(5.2)$ and $(5.3)$  respectively.
Indeed, $\mbox{Ric} \ge - \tau \rho = \frac {n-2}{4(n-1)}\tau R\, = a R\, ,$ $(5.2) \Longleftrightarrow (5.8)\, ,$ in which $a=\frac {n-2}{4(n-1)}\tau\, ,$ and $$(5.3)\quad \tau <\frac{p-1+\kappa}{p^{2}}\quad \Longleftrightarrow \quad (5.9)\quad a <\frac {n-2}{4(n-1)} \cdot \frac{p-1+\kappa}{p^{2}}.$$ Now the assertion follows immediately from Theorem A. \end{proof}
\begin{theorem}
\label{T2} Let $(M^n, g)$, $n\geq 3$, be a complete non-compact, simply connected, locally conformal flat Riemannian manifold with scalar curvature $R < 0\, ,$ a.e. and Ricci curvature satisfying
\begin{equation}
\begin{array}{lll}
\mbox{Ric}(x)  \geq  a R(x)%
\end{array}
\label{Rs}
\end{equation}%
for all $x\in M,$ where $a $ is a constant such that
\begin{equation}
\begin{array}{l}
a <\frac {n-2}{n-1} \cdot \frac{ q-1+\kappa +b }{q^{2}},\text{ }%
\end{array}
\label{1.7}
\end{equation}%
in which $p>1\, ,$ and $\kappa$ is as in $(5.7)\, .$

Let $u\in C^{3}\left( M\right) $ be a strongly $p$-harmonic function with
finite $q$-energy $E_{q}\left( u\right) <\infty .$ \newline
Then the conclusions $(I)$ and $(II)$ as in Theorem B hold.
\end{theorem}
\begin{proof}
Arguing as before, the inequalities $(5.10)$ and $(5.11)$ are equivalent to the inequalities $(5.5)$ and $(5.6)$  respectively, and the assertion follows immediately from Theorem B.
\end{proof}

\section{Topology of LCF Riemannian manifolds}
According to the vanishing theorem in Sect.4, we can study the topology at infinity of LCF manifolds.
\begin{theorem}
\label{one-end-Ric}
 Let $(M^n, g)$, $n \geq 3$,  be a complete, simply connected,    locally conformally flat Riemannian manifold.
Then there exists a constant $C>0$ such that if
\begin{equation}
\int_M |\mbox{Ric}|^{\frac{n}{2}}dv< C, \label{Lp-Ric2}
\end{equation}
then $M$ has only one end.
\end{theorem}
\begin{proof}
By the hypothesis, it follows from Theorem \ref{thm-Ric} that $H^1(L^2(M))=\{0\}$.
The assumption (\ref{Lp-Ric2}) implies that the following Sobolev inequality
\begin{equation*}
C_s \big(\int_M |f|^{\frac{2n}{n-2}}dv \big)^{\frac{n-2}{n}}  \leq  \int_M |\nabla f|^2dv, \ \ \forall f \in C^\infty_0(M)
\end{equation*}
holds for some $C_s>0$. Hence $M$ has infinite volume. According to Corollary 4 of \cite{LW1},
each end of $M$ is non-parabolic. By the important result in \cite{LT}, the number of non-parabolic ends of $M$ is at most the dimension of the space of  harmonic functions with finite Dirichlet integral.
 Observe that if $f$ is a harmonic function with finite Dirichlet integral then its exterior $df$ is an $L^2$ harmonic $1$-form. Therefore,  $M$ has only one end.
 \end{proof}

Considering the case of $p=1$ in Theorem \ref{thm-global-pinching}, using an analogous method as above, we  have
\begin{theorem}
\label{one-end-E}
 Let $(M^n, g)$, $n\geq 3$, be a complete non-compact, simply connected, locally conformally flat Riemannian manifold with $R \geq 0$. Assume that
 \begin{equation}
\Big(\int_M |E|^{\frac{n}{2}}dv\Big)^{\frac{2}{n}} < C(n), \label{gap2}
\end{equation}
where $C(n)=\frac{(n-2)\sqrt{n}}{|n-2|\sqrt{n-1}}  \min \Big\{\frac{n}{n-1}, \frac{4(n-1)}{n(n-2)}\Big\} Q(\mathbb{S}^n)$, then $M$ has only one non-parabolic end.
\end{theorem}

\begin{remark}
In $\cite{Li}$, H.Z. Lin proved a  one-end theorem for LCF manifolds by assuming that $R \leq 0$ and $(\int_M |E|^{n}dv)^{\frac{2}{n}} < C(n)$ for some explicit constant $C(n)>0$.

\end{remark}

From Theorem \ref{thm-pinching} and the Sobolev inequality (\ref{R-negative1}), we have the following one end theorem under pointwise condition.
\begin{theorem}
Let $(M^n, g)$, $n\geq 4$, be a complete, simply connected, locally conformally flat Riemannian manifold with $R \leq 0$.  Suppose  that
\begin{equation}
\mbox{Ric}(x) \geq \frac{n-2}{4(n-1)} R(x)
\end{equation}
for all $x \in M$.  Then   $M$ has only one end.
\end{theorem}
\begin{proof}
Suppose contrary, there were at least two ends, then by the method
in [27, p.681-683], there would exist a nonconstant bounded harmonic function $f$
with finite energy on $M$. Hence $df$ would be a nonconstant $L^2$ harmonic $1$-form on $M$. That is,
 $H^1(L^2(M))\neq \{0\}$, contradicting Theorem \ref{thm-pinching} in which $p=1$.
\end{proof}

\begin{remark}
In $\cite{LW2}$, Li-Wang proved that for a complete, simply connected, LCF manifold $M^n$ $(n\geq 4)$ with  $R \leq 0$, if the Ricci curvature $\mbox{Ric} \geq \frac{1}{4} R$ and the scalar curvature satisfies some decay condition, then either  $M$ has only one end, or   $M= \mathbb{R}\times N$ with a warped product metric for some compact manifold $N$.
\end{remark}

\noindent Yuxin Dong\\
 School of Mathematical Science, Fudan University, Shanghai, 200433,  China.\\
E-mail: yxdong@fudan.edu.cn\\

\noindent Hezi Lin\\
 School of Mathematics and Computer Science,
 Fujian Normal University, Fuzhou,  350108,
 China.  E-mail: lhz1@fjnu.edu.cn\\

\noindent Shihshu Walter Wei\\
Department of Mathematics, University of Oklahoma, Norman, Oklahoma 73019-0315, USA. E-mail: wwei@ou.edu

%\end{CJK*}

\begin{thebibliography}{99}

\bibitem{Ba} P. Baird, Stress-energy tensors and the Lichnerowicz Laplacian, J. Geom.  Phys.  58 (2008), 1329-1342.

\bibitem{Bo} J.P. Bourguignon,  Les vari\'{e}t\'{e}s de dimension 4 \`{a} signature non nulle dont la courbure est harmonique sont d'Einstein, Invent. Math. 63(2) (1981),  263-286.

\bibitem{Ca} G. Carron,  Une suite exacte en $L^2$-cohomologie, Duke Math. J. 95(2) (1998), 343-372.

\bibitem{CCW} S.C. Chang, J.T. Chen and S.W. Wei,  Liouville properties for $p$-harmonic maps with finite $q$-energy, Trans. Amer. Math. Soc. 368(2) (2016),  787-825; arXiv:1211.2899.


\bibitem{CGH} D.M.J. Calderbank, P. Gauduchon and M. Herzlich,  Refined Kato inequalities and conformal weights in Riemannian geometry, J. Funct. Anal.  173(1)  (2000),   214-255.

\bibitem{CH} C. Carron and M. Herzlich, The Huber theorem for non-compact conformally flat manifolds, Comment. Math. Helv. 77 (2002), 192-220.

\bibitem{Che} Q.M. Cheng,  Compact locally conformally flat Riemannian manifolds, Bull. London Math. Soc. 33 (2001), 459-465.

\bibitem{CZ} B.L.  Chen and X.P. Zhu,  A gap theorem for complete noncompact manifolds with nonnegative Ricci curvature, Comm. Anal. Geom.  10(1) (2002), 217-239.

\bibitem{DL} Y.X. Dong and H.Z. Lin,  Monotonicity formulae, vanishing theorems and some geometric applications,  Quart. J. Math. 65 (2014), 365-397.

\bibitem{DW} Y.X. Dong and S.W. Wei,  On vanishing theorems for vector bundle valued $p$-forms and their applications, Comm. Math. Phy. 304(2) (2011), 329-368.


\bibitem{GW} R.E. Greene and H. Wu, Function theory on manifolds which possess a pole, Lecture Notes in Math. Vol.699, Springer-Verlag, 1979.

\bibitem{GLW}  P.F. Guan, C. S. Lin and G.F. Wang, Schouten tensor and some topological properties, Comm. Anal. Geom. 13(5) (2005), 887-902.

\bibitem{HLRW} Y.B Han, Y. Li, Y.B. Ren and S.W. Wei,  New comparison theorems in Riemannian geometry, Bull. Inst. Math. Acad. Sin. (N.S.) 9 (2014), no. 2, 163-186.

\bibitem{Li} H.Z. Lin, On the structure of conformally flat Riemannian manifolds,  Nonlinear Anal. 123-124 (2015), 115-125.

\bibitem{LT} P. Li and L.F. Tam, Harmonic functions and the structure of complete manifolds, J. Diff. Geom.
  35 (1992), 359-383.

\bibitem{LW1} P. Li and J.P. Wang, Minimal hypersurfaces with finite index, Math. Res. Lett. 9 (2002), 95-103.

\bibitem{LW2} P. Li and J.P. Wang,  Weighted Poincar\'{e} inequality and rigidity of complete manifolds, Ann. Scient. \'{e}c. Norm. Sup. 39 (2006), 921-982.

\bibitem{MN}  F. Mercuri and M.H. Noronha,  Low codimensional submanifolds of Euclidean space with nonnegative isotropic curvature, Trans. Amer. Math. Soc. 348 (1996), 2711-2724.

\bibitem{MW} M. Micallef and M.Y. Wang,  Metrics with nonnegative isotropic curvature, Duke Math. J. 72
       (1993), 649-672.

\bibitem{PV} S. Pigola and G. Veronelli,  Remarks on $ L^{p} $-vanishing results in geometric analysis, Internat. J. Math. 23 (2012), no. 1, 1250008, 18 pp.

\bibitem{PRS1} S. Pigola, M. Rigoli and  A.G. Setti,  Some characterizations of space-forms, Trans. Amer. Math. Soc. 359(4) (2007), 1817-1828.

\bibitem{PRS2} S. Pigola, M. Rigoli and  A.G. Setti,  Volume growth, ``A priori" estimates, and geometric applications, Geom. Funct. Anal. 13 (2003), 1302-1328.

\bibitem{PST} S. Pigola,   A.G. Setti and M. Troyanov, The connectivity at infinity of a manifold and
          $L^{p, q}$-Sobolev inequalities, Expo. Math 32 (2014), 365-383.

\bibitem{SY} R. Schoen and S.T. Yau,  Conformally flat manifolds, Kleinian groups and scalar curvature, Invent. Math. 92 (1988), 47-71.

\bibitem{SW} E.M. Stein  and  G. Weiss,  Introduction  to  Fourier  Analysis  on  Euclidean  Spaces,
        Princeton Math. Series 32, Princeton University Press, Princeton, NJ, 1971.

\bibitem{W1} S.W. Wei, $p$-harmonic geometry and related topics, Bull. Transilv. Univ.
Brasov, Ser. III 1(50), 415-453. 2008.

\bibitem{W2} S.W. Wei, The structure of complete minimal submanifolds in complete
manifolds of nonpositive curvature, Houston J. Math. 29(3) (2003),  675-689.

\bibitem{Wu} H. Wu,  The Bochner technique in differential geometry, Mathematical Reports, Vol 3, Pt
2, Harwood Academic Publishing, London, 1987.

\bibitem{Xi}  Y.L. Xin, Differential forms, conservation law and monotonicity formula, Scientia Sinica (Ser A) Vol.XXIX (1986), 40-50.

\bibitem{XZ} H.W. Xu and E.T. Zhao,  $L^p$ Ricci curvature pinching theorems for conformally flat Riemannian manifolds, Pacific J. Math. 245(2) (2010), 381-396.

\bibitem{Ya} S.T. Yau,  Some function-theoretic properties of complete Riemannian manifold and their
     applications to geometry, Indiana Univ. Math. J. 25 (1976), 659-670. Erratum: ``Some function-theoretic properties of complete Riemannian manifold and their applications to geometry" [Indiana Univ. Math. J. 25 (1976), no. 7, 659-670]. Indiana Univ. Math. J. 31 (1982), no. 4, 607.

\bibitem{Zh} S.H. Zhu, The classification of complete locally conformally flat manifolds of nonnegative Ricci curvature, Pacific J. Math. 163(1) (1994), 189-199.


\end{thebibliography}
\end{document}